\documentclass[12pt, a4paper]{article}
\usepackage{amssymb,amsmath,latexsym,color,graphicx}
\usepackage{hyperref}
\newtheorem{theorem}{Theorem}[section]
\newtheorem{lemma}[theorem]{Lemma}
\newtheorem{proposition}[theorem]{Proposition}

\newtheorem{property}[theorem]{Property}

\newtheorem{remark}[theorem]{Remark}

\newenvironment{proof}{\par\noindent{{\bf Proof.}}}{\hfill$\Box$
\medskip}

\title{Stability estimates for  semigroups \\    in the Banach case.\\
}
\author{B.~Helffer.}

\date{\today}
\begin{document}

\bibliographystyle{plain}

\maketitle
\begin{abstract}
The purpose of this paper is to revisit  previous works of the author with J. Sj\"ostrand (2010--2021) proved  in the Hilbert case by considering the Banach case  at the light of a paper by Y.~Latushkin and V.~Yurov (2013).
\end{abstract}
\section{Introduction}\label{int}\setcounter{equation}{0}

Let ${\cal B}$ be a complex  Banach space and let $[0,+\infty [\ni
t\mapsto S(t)\in {\cal L}({\cal B},{\cal B})$ be a strongly continuous
semigroup with $S(0)=I$. 

\par The purpose of this note is to revisit some results of \cite{HelSj} and \cite{HelSj21} which have been  established in the case of a Hilbert space and to consider their extension to the Banach case.
 The idea in the Hilbert space is essentially to use the property of 
  the inhomogeneous equation $(\partial _t-A)u=w$ in exponentially
  weighted spaces which are related via
  Fourier-Laplace transform and then use Plancherel's formula. This approach cannot be used in the Banach case (see \cite{Pe} for a discussion) but one can see in \cite{LY} an alternative approach proposed by Y. Latushkin and V. Yurov which combined with the approach of \cite{HelSj} permits this extension. We will describe how to extend also all the results of \cite{HelSj21}.

\par The main result in \cite{HelSj} was  established in the Hilbertian case:
\begin{theorem}\label{int2}~\\ Let assume that $\mathcal B$ is a Hilbert space.
Let $r(\omega)$ defined by
$$
\frac{1}{r(\omega)}:=\sup_{\Re z \geq \omega} ||(A-z)^{-1}||_{\mathcal L(\mathcal B)}\,.
$$
Let $m(t)\ge \Vert S(t)\Vert_{\mathcal L(\mathcal B)} $ be a continuous positive function.
 Then
for all $t,a,b>0$, such that $t \geq  a+b$, we have
\begin{equation}\label{int.4}  \Vert S(t)\Vert \le \frac{e^{\omega t}} {r(\omega )\Vert
    \frac{1}{m}\Vert_{e^{-\omega \cdot }L^2(0,a)}\Vert
    \frac{1}{m}\Vert_{e^{-\omega \cdot }L^2(0,b)}}.
    \end{equation}
\end{theorem}
Here the norms are always the natural ones obtained from ${\cal B}$,
$L^2$, thus for instance 
$\Vert S(t)\Vert= \Vert S(t)\Vert_{{\cal L}({\cal B},{\cal B})}$, if
$u$ is a function on $\mathbb{R}$ with values in $\mathbb{C}$ or in
${\cal B}$, $\Vert u\Vert$  denotes the natural $L^2$ norm. In (\ref{int.4}) we also have the natural norm
in the exponentially weighted space $e^{-\omega \cdot }L^2(0,a)$ and
similarly with $b$ instead of $a$; $\Vert
f\Vert_{e^{-\omega \cdot }L^2(0,a)}=\Vert e^{\omega \cdot }f(\cdot
)\Vert_{L^2(0,a)}$. \\
 The extension proposed by \cite{LY} (with less generality\footnote{The authors consider the case when $m(t):= Le^{\lambda t} $ and play instead with the consequence of this theorem as in \cite{HelSj}.}) can be obtained by introducing    $K_{\omega,p}$ which  is defined by
  \begin{equation}\label{eq:K}
 \frac{1}{\hat r_p(\omega)}:=  K_{\omega,p}:= ||\mathcal K^+||_{\mathcal L (L^p_\omega(\mathbb R_+;X))}<+\infty
   \end{equation}
 where 
\begin{equation}\label{eq:Ka}
(\mathcal K ^+ u)(t)=\int_0^t\, S(t-s) u(s) ds\,,\, \mbox{ for } t\geq 0\,.
\end{equation}
Our first new theorem  reads:
 \begin{theorem}\label{thBanach}
 Suppose that  $\omega \in \Bbb{R}$, $p>1$,  and that $K_{\omega,p}$ is  finite. 
Let $m(t): [0, +\infty[ \to ]0, +\infty[$ be a continuous positive function such that 
\begin{equation}\label{eq:h1}
\|S(t)\| \leq m(t) \mbox{ for all } t \geq 0\,.
\end{equation}
Then for all $t,a,b >0$ such that $t \geq  a+ b  $, 
\begin{equation}\label{int.4aaa}  \Vert S(t)\Vert \leq K_{\omega,p}  \frac{e^{\omega t}}{ \Vert
    \frac{1}{m}\Vert_{e^{-\omega \cdot }L^q(]0,a[)}\Vert
    \frac{1}{m}\Vert_{e^{-\omega \cdot }L^p(]0,b [)} }\,,
    \end{equation}
    where $q$ is such that $\frac 1p + \frac 1q =1$\,.
    \end{theorem}
The idea behind this extension by \cite{LY} is to replace the use of the Laplace transform and the assumptions on the resolvent of $A$ by more directly the estimate for $\mathcal K^+$. Hence $1/r(\omega)$ is replaced by $K_{\omega,p}$.\\ 
The proof will be given in Section \ref{s2}.

Another aim in this note is to see if we can also extend to the Banach case the results of \cite{HelSj21}. Some of the steps are independent of the Banach assumption, so we  just explain the modifications to do in the case of a reflexive Banach space.\\ 
 Let $\Phi$ satisfy
\begin{equation}\label{defpropchi}
0\le \Phi \in C^1([0,+\infty [) \mbox{  with } \Phi (0)=0 \mbox{ and }  \Phi (t)>0
\mbox{ for } t>0\,,
\end{equation}
and assume that $\Psi$ has the same properties\footnote{By a density argument we can replace $C^1([0,+\infty [)$ in
(\ref{defpropchi}) by the space of locally Lipschitz functions on
$[0,+\infty [$.}.  For $t>0$, let $\iota_t$
be the reflection with respect to $t/2$: $\iota _tu(s)=u(t-s)$. With
this notation, we have the following theorem:
 \begin{theorem}\label{Th1.2p} We assume that $\mathcal B$ is a complex reflexive Banach space. Under the assumptions of Theorem
   \ref{thBanach},  for any $\Phi$ and $\Psi$ satisfying
   \eqref{defpropchi} and for any $\epsilon_1,\epsilon_2 \in \{-,+\}$, we have 
\begin{equation}\label{LL.5bis}
 || S (t)|| _{\mathcal L(\cal B)}\le  e^{\omega t}\frac{\|(\hat r_p(\omega
  )^p\Phi^p -|\Phi '|^p)^{\frac{1}{p}}_-\,m\|_{e^{\omega \cdot }L^p(]0,t[) } \|(\hat r_p(\omega)^q \Psi
  ^q- |\Psi '|^q)^{\frac{1}{q}}_-\, m\|_{e^{\omega \cdot }L^q(]0,t[) }}
{\int_0^t (\hat r_p(\omega
    )^p\Phi^p -|\Phi '|^p)^{\frac{1}{p}}_{\epsilon_1}(\hat r_p(\omega )^q (\iota_t \Psi)^q-|\iota_t \Psi
    '|^q)^{\frac{1}{q}}_{\epsilon_2} ds}\,.
\end{equation}
\end{theorem}
Here for $a\in \mathbb R$, $a_+=\max (a,0)$ and  $a_-= \max (-a,0)$. This theorem was established in \cite{HelSj21} (Theorem 1.6)  in the Hilbert case with $p=2$ and $\hat r_2(\omega)$ replaced by $r(\omega)$. \\
With this generalization, the other sections of \cite{HelSj21} hold in the case $p=2$, in particular all the consequences considered as 
$\epsilon_1 \epsilon_2 =-$ in \cite{HelSj21} (Theorem 1.7, Proposition 1.8, Theorem 1.9). Notice that in the Hilbert case the replacement of $r(\omega)$  by $\hat r_2(\omega)
 = 1 / K_{\omega,2}$ which could appear as stronger but is probably equivalent (see the discussion in \cite{LY}).The case $p\neq 2$ will be further discussed in the last sections, with in particular an extension of Proposition 1.8 (\cite{HelSj21} and an extension of Wei's theorem.\\
 These applications  will be stated  in Theorem \ref{th3.2p}. and in Theorem \ref{prop4.9bis}.\\
  \begin{remark}
 There is a huge litterature on the subject and we refer for example to the recent survey by \cite{Ro} for a description of the different approachs. In particular J. Rozendaal and M. Veraar \cite{RoVe} obtain more general results as here (see for example their Theorem 3.2) albeit with different constants when considering our particular case. The question of the finiteness
  of $K_{\omega,p}$ is also an interesting open question. Notice\footnote{We thank J. Rozendahl for this remark} that the inequality $K_{\omega,p} \leq 1/r(\omega)$ also holds, by application of Proposition 3.7 in \cite{RoVe} when $S$ is a positive semi-group on an $L^p$-space.
 \end{remark}

{\bf Acknowledgements}\\
The author was motivated by  a question of L. Boulton 
 at the Banff conference (2022) and encouraged  later by discussions with Y. Latushkin during the Aspect Conference in Oldenburg (2022) organized by K. Pankrashkin.\\ We thank J. Rozendaal for comments on a previous version of the manuscript and   Y. Latushkin for his careful reading and suggestions of simplification.
 
\section{Proof of Theorem \ref{thBanach} and applications.}\label{s2}
\subsection{Proof of Theorem \ref{thBanach}}
    Following \cite{HelSj} and a simplification proposed by Y. Latushkin, we
   consider  $v\in D(A)$ and  $u(t)=S(t) v$, for solving the Cauchy problem
$$
\begin{array}{l}
(\partial_t -A) u=0\,,\, t\geq 0\,\,,\\
u(0)=v\,.
\end{array}
$$
Assume $t> a+b$, and let   $\chi$, $\tilde \chi$  be decreasing Lipschitz functions on $\mathbb R$, equal to
$1$ on $]-\infty,0]$  and such that ${\rm supp\,}  \chi \in (-\infty,a)$ and ${\rm supp\,}  \tilde \chi \in (-\infty,b)$.
We notice that
\begin{equation}\label{eq:L1}
u(t) = S(t-s) S(s) v= S(t-s) u(s)\,,\, \mbox{ for } t\geq s \geq 0\,.
\end{equation}
Then, using \eqref{eq:L1} between the first line and the second line, we obtain, for $t\geq 0$, 
$$
\begin{array}{ll}
(1- \chi(t) )  u(t) &=- (\int_0^t \chi'(s) ds)\, u(t)\\ 
&=  -\int_0^t S(t-s) \chi'(s) u(s) ds \\
&=- \mathcal K^+ (\chi'(\cdot)u(\cdot)) (t)\,.
\end{array}
$$
On the other hand, introducing $\tilde \chi$, we first write
$$
u(t) = \tilde \chi (0) u(t) = - \big( \int_{t-b}^t \tilde \chi'(t-s) ds\big) u(t)\,.
$$
Proceeding similarly as above we then write (using that $t-b>a$  at the third line)
$$
\begin{array}{ll}
u(t)&=  -  \int_{t-b}^t \tilde \chi'(t-s) S(t-s) u(s) ds  \\
& =  -  \int_{t-b}^t \tilde \chi'(t-s) S(t-s)  u(s) ds \\
& =  -  \int_{t-b}^t \tilde \chi'(t-s) S(t-s) (1-\chi(s)) u(s) ds \\
&=  -  \int_{t-b}^t \tilde \chi'(t-s) S(t-s) \mathcal K^+ (\chi'(\cdot)u(\cdot)) (s) ds \,.
\end{array}
$$
Thus, with in mind the assumption on $||S(t)||$, we get
$$
|| e^{-\omega t} S(t)v || \leq \Big(\int_{t-b}^t  \big(e^{-\omega (t-s)}|\tilde \chi'(t-s)| m(t-s)\big) \big( || \mathcal K^+ (\chi'(\cdot)u(\cdot)) (s)||\big) ds \Big)\,,
$$
and H\"older's inequality yields 
  \begin{equation} \label{eq:banach}
   e^{-\omega t} || S(t)||_{\mathcal L (\mathcal B)}\leq  K_{\omega,p} || m \chi'||_{e^\omega L^p  }  ||e^{-\omega s} m \tilde  \chi'||_{e^\omega L^q  }\,,
\end{equation}
as 
$$
 || \mathcal K^+ (\chi'(\cdot)u(\cdot))||_{e^{\omega \cdot} L^p}\leq K_{\omega,p} ||\chi'(\cdot)S(\cdot) v||_{e^{\omega \cdot} L^p}\
  \leq K_{\omega,p} || \chi' m ||_{e^{\omega \cdot} L^p} ||v||\,.
$$
 This is  the inequality  (3.8) in \cite{LY}  but with the general weight $m(t)$ replacing the particular weight  $L e^{\lambda t}$.
 When $p=2$, \eqref{eq:banach} is just the Banach analog of (4.12) in \cite{HelSj} where $1/r(\omega)$ is replaced by $K_{\omega,2}$.\\
 Following the strategy of \cite{HelSj}  it remains to optimize the right hand side by choosing $\chi$ and $\tilde \chi$ optimally.    We look for $\chi$  such that $\| m
\chi'\|_{e^\omega L^p(0,a)}$
 is as small as possible. By the H\"older inequality, 
\begin{equation}\label{18}
1 = \int_0^a |\chi'(s) | ds \leq \|\chi' m\|_{e^{\omega \cdot}L^p}
\|\frac 1m\|_{e^{-\omega \cdot}L^q(]0,a[)}\,,
\end{equation}
so
\begin{equation}\label{19}
 \|\chi' m\|_{e^{\omega \cdot}L^p}  \geq \frac{1}{\|\frac
   1m\|_{e^{-\omega \cdot}L^q(]0,a[)}}\,.
\end{equation}
As classical, we get equality in \eqref{18} if for some constant $C$,
$$
(|\chi'(s)| m(s) e^{-\omega s})^p = C  \big( \frac{1}{m(s)} e^{\omega s}\big)^q \hbox{
  on }[0,a],
$$
i.e.
$$
\chi'(s) m(s) e^{-\omega s} = - C^{1/p} \Big( \frac{1}{m(s)} e^{\omega s}\Big)^{q/p} \hbox{
  on }[0,a],
$$
where $C$ is determined  by the condition
$$
1 = \int_0^a |\chi'(s) | ds\,.
$$

Doing the same job with $\tilde \chi$, we obtain the theorem.

\subsection{The result of Latushkin-Yurov and extensions}
In \cite{LY}, the authors prove directly the following statement:
\begin{theorem}
Let $\omega, \lambda$, $p>1$ and $L>0$.
Let $\{S(t)\}_{t\geq 0}$ be a strongly continuous semigroup on a Banach space $\mathcal B$. 
If $\omega < \lambda$, $|| S(t)|| \leq L e^{\lambda t} \mbox{ for all } t\geq 0$
 and $K_{\omega,p} < +\infty$\,,
  then 
 $$
 || S(t)|| \leq M e^{\omega t} \mbox{ for all } t\geq 0\,,
 $$
 with 
 $$
 \frac 1p + \frac 1q=1\,,
 $$
 and
 $$
 M= L (1 + 4p^{-1/p} q^{-1/q} L K_{\omega,p} (\lambda -\omega))\,.
 $$
 \end{theorem}
 The theorem can also be obtained as in \cite{HelSj} for the case $p=2$ as a corollary of Theorem \ref{thBanach}.

When $p=2$ and $\mathcal B$ is an Hilbert space  it is possible to prove (see \cite{LY})  that
$$
K_{\omega,2}\leq  \sup_{s\in \mathbb R} || R(A,\omega+ is)||_{\mathcal L(\mathcal B)}\,,
$$
where $R(A,\lambda) = (\lambda - A)^{-1}$ and we recover the statement of  \cite{HelSj21}, which was obtained as a consequence of the $L^2$ version of Theorem \ref{thBanach} with $p=2$.
 
 \section{Proof of Theorem \ref{Th1.2p} in the reflexive Banach case.}\label{s3} \setcounter{equation}{0}
\subsection{Flux} We assume that $\mathcal B$ is a reflexive Banach space and we denote by $\mathcal B^*$ its dual\footnote{The more common notation is $\mathcal B'$.}. As before, let $A$ the generator of a strongly continuous semi-group and 
$u(t)\in C^1([0,+\infty [;{\cal B})\cap C^0([0,+\infty [;{\cal
  D}(A))$ solve $(A-\partial _t)u=0$ on $[0,+\infty [$.\\  As $\mathcal B$ is reflexive, we can define  $A^*$ as  the infinitesimal generator of the dual semi-group which is a strongly continuous semi-group on $\mathcal B^*$. Let 
$u^*(t)\in C^1(]-\infty ,T];{\cal B^*})\cap C^0(]-\infty ,T];{\cal
  D}(A^*))$ solve   $(A^*+\partial _t)u^*=0$
 on $]-\infty ,T]$. 
We refer to \cite{AE} (and references therein) for the properties of the  dual semi-group. 
 Then the flux (or
Wronskian) $<u(t),u^*(t)>_{\mathcal B,\mathcal B^*}$ (where the bracket indicates the duality bracket between $\mathcal B$ and $\mathcal B^*$) is constant on $[0,T]$ as can be seen by
computing the derivative with respect to $t$.

\subsection{$L^p$ estimate} 
Write $L^p_\phi (I)=L^p(I;e^{-p\phi }dt)=e^\phi L^p(I)$, $\|u\|_{p,\phi}
=\|u\|_{p,\phi ,I}=\|u\|_{L^p_\phi (I)}$, where $I$ is an interval and our functions take values in ${\cal
  B}$. We assume (see \eqref{eq:K}) that 
$K_{p,\omega} = 1/\hat r_p(\omega) < +\infty$.
\par Consider $(A-\partial _t)u=0$ on $[0,+\infty [$ with
$u\in L^p_{\omega \cdot }([0,+\infty [)$. \\
Let $\Phi$ satisfy (\ref{defpropchi}) and add temporarily the
assumption that $\Phi (s)$ is constant for $s\gg 0$. 
 Then $\Phi u$, $\Phi 'u$ can be viewed as elements of
$L^p_{\omega \cdot }(\mathbb R)$ and from
$$
(A-\partial _t)\Phi u=-\Phi 'u\,,
$$
we get, by the definition of $\hat r_p(\omega)$,  
$$
\| \Phi u\|_{p,\omega \cdot }\le \frac{1}{\hat r_p(\omega )}\|\Phi
'u\|_{p, \omega \cdot }\,.
$$
Taking the power $p$, we get
$$  
\int_{-\infty}^{+\infty} (\hat r_p(\omega )^2|\Phi|^p -|\Phi '|^p) ||u(t)||^p_\mathcal B e^{-p\omega t} dt \leq 0\,,
$$
which  can be rewritten as 
\begin{equation*}
\int_{-\infty}^{+\infty} (\hat r_p(\omega )^2|\Phi|^p -|\Phi '|^p)_+ ||u(t)||^p_\mathcal B e^{-p\omega t} dt \leq  \int_{-\infty}^{+\infty} (\hat r_p(\omega )^p |\Phi|^p -|\Phi '|^p)_- ||u(t)||^p_\mathcal B e^{-p\omega t} dt\,,
\end{equation*}
or finally in. the form
\begin{equation}\label{L2.2p}
\| (\hat r_p(\omega )^p|\Phi|^p -|\Phi '|^p)_+^{1/p }u\|_{p, \omega \cdot }\le 
\| (\hat r_p(\omega )^p|\Phi|^p -|\Phi '|^p)_-^{1/p}u\|_{p,\omega \cdot }.
\end{equation}
Writing $\Phi =e^{\phi }$, $\phi \in C^1(]0,+\infty [)$, $\phi (t)\to
-\infty $ when $t\to 0$, we have
$$
\hat r_p(\omega )^p|\Phi|^p-|\Phi '|^p= (\hat r_p(\omega )^p-|\phi '|^p)e^{p\phi } \,,
$$
and (\ref{L2.2p}) becomes
\begin{equation}\label{L2.4p}
\| \big(\hat r_p(\omega ))^p -|\phi '|^p\big)_+^{1/p}u\|_{p,\omega \cdot -\phi} \le 
\| \big(\hat r_p(\omega )^p -|\phi '|^p\big)_-^{1/p}u\|_{\omega \cdot-\phi }\,.
\end{equation}
 Let $S(t)=e^{tA}$, $t\ge 0$ and let $m(t)>0$ be a continuous function such
that
\begin{equation}\label{L2.5p}
\| S(t)\|\le m(t),\ t\ge 0\,.
\end{equation}
Then we get
\begin{equation}\label{L2.6p}
  \| \big(\hat r_p(\omega)^p-|\phi '|^p\big)_+^{1/p}u\|_{p,\omega \cdot -\phi  }
  \le \| \big(\hat r_p(\omega)^p -|\phi '|^p\big)_-^{1/p}m\|_{p,\omega \cdot -\phi 
  }|u(0)|_{\cal B}\,.
\end{equation}
 Note that we have also trivially
\begin{equation}\label{L2.6bisp}
  \| (\hat r_p(\omega )^p-|\phi '|^p)_-^{1/p}u\|_{p,\omega \cdot -\phi  }
  \le \| \big(\hat r_p(\omega )^p-|\phi '|^p\big)_-^{1/p}m\|_{p,\omega \cdot -\phi 
  }|u(0)|_{\cal B}\,.
\end{equation}

\par  We get the same bound for the forward solution of $A^*-\partial
_t$ and, after changing the orientation of time, for the backward
solution of $A^*+\partial _t=(A-\partial _t)^*$. Then for $u^*(s)$,
solving
$$
(A^*+\partial _s)u^*(s)=0,\ s\le t,
$$
with $u^*(t)$ prescribed in $\mathcal B^*$, we get
\begin{equation*}
  \| \big(\hat r_q^*(\omega )^q-|\iota_t\phi '|^q\big)_+^{1/q}u^*\|_{q,\omega (t-\cdot )-\iota_t \phi }
  \le
  \| \big(\hat r_q^*(\omega )^q-|\iota_t \phi '|^q\big)_-^{1/q}\iota_t m
  \|_{q,\omega (t-\cdot )-\iota_t \phi }\,  |u^*(t)|_{\mathcal B^*}\,,
\end{equation*}
where $\iota_t \phi $ and $\iota_t m $ denote the
compositions of $\phi $ and $m$ respectively with the reflection $\iota_t$  in
$t/2$ so that $$\iota_t m (s)=m(t-s),\ \ \iota_t \phi (s)= \phi (t-s)\,.$$
Here $\hat r_q^*(\omega)$ is associated with the dual semi-group like in \eqref{eq:K}-\eqref{eq:Ka}. By duality, one can 
show in the reflexive case that
\begin{equation}\label{eq:3.7}
\hat r_q^* (\omega) = \hat r_p(\omega)\,.
\end{equation}

 More generally, we can replace $\phi$ by $\psi$ with the same properties (see \eqref{defpropchi})
and consider $ \Psi  = \exp \psi\,.
$

 Note that we have 
\begin{equation}\label{L2.7p}
  \| \big(\hat r_q^*(\omega )^q-|\iota_t \psi '|^q\big)_+^{1/q}u^*\|_{q, \omega (t-\cdot
    ) -\iota_t \psi  }
  \le \| \big(\hat r_q^*(\omega )^q-|\psi'|^q\big)_-^{1/q}m\|_{q,\omega \cdot -\psi 
  }|u^*(t)|_{\cal B^*}.
\end{equation}
and  also trivially
\begin{equation}\label{L2.7bisp}
  \| \big(\hat r_q^*(\omega )^q-|\iota_t \psi '|^q\big)_-^{1/q}u^*\|_{q, \omega (t-\cdot )-\iota_t \psi }
  \le \| \big(\hat r_q^*(\omega )^q-|\psi'|^q\big)_-^{1/q}m\|_{q,\omega \cdot -\psi 
  }|u^*(t)|_{\cal B^*}.
\end{equation}

\subsection{From $L^p$ to $L^\infty $ bounds}\label{LL} 
 In order to estimate
$|u(t)|_{{\cal B}}$ for a given $u(0)$ it suffices to estimate \break 
 $|<u(t),u^*(t)>_{\cal B, \cal B^*}|$ for arbitrary $u^*(t)\in {\cal B}^*$. 
 Extend
$u^*(t)$ to a backward solution $u^*(s)$ of  $(A^*+\partial
_s)u^*(s)=0$, so that
\begin{equation} \label{LL.0p} 
<u(s), u^*(s)>_{\cal B,\cal B^*}= <u(t), u^*(t)>_{\cal B,\cal B^*},\ \forall s\in [0,t]. 
\end{equation}
Let $M=M_t:[0,t]\to [0,+\infty [$ have mass 1:
\begin{equation}\label{LL.1p}
\int_0^t M(s)ds=1\,.
\end{equation}
Then, using \eqref{LL.0p},
\begin{multline}\label{LL.1.5p}  
 | <u(t), u^*(t)>_{\cal B,\cal B^*}| =\left| \int_0^t M(s) | <u(s), u^*(s)>_{\cal B,\cal B^*}|   ds \right| \\
  \le \int_0^t M(s) |u(s)|_{\cal B}|u^*(s)|_{\cal B^*} ds. \end{multline}
 Let $\epsilon _1,\epsilon_2 \in \{-,+\}$.  Assume that
\begin{equation}\label{LL.2p}
  \mathrm{supp\,}M\subset \{ s; \epsilon_1(\hat r_p(\omega)^p-|\phi '(s)|^p) > 0,\
\epsilon_2(  \hat r_q^*(\omega)^2-\iota_t|\psi '(s)|^q) > 0  \}. 
\end{equation}

Then multiplying and dividing with suitable factors in the last member
of (\ref{LL.1.5p}), we get (in the reflexive case) 
 
\begin{equation*} 
\begin{split}
| <u(t), u^*(t)>_{\cal B,\cal B^*}|& \le e^{\omega t}\int_0^t
  \frac{M(s)e^{-\phi (s)-\iota_t \psi (s)}} {(\hat r_p(\omega)^p-|\phi
    '(s)|^p)^{\frac{1}{p}}_{\epsilon_1}(\hat r_q^*(\omega)^q-|\iota_t \psi
    '(s)|^q)^{\frac{1}{q}}_{\epsilon_2}}\times\\ 
    & \qquad \times  e^{\phi (s)-\omega s} (\hat r_p(\omega
  )^p-|\phi '(s)|^p)^{\frac{1}{p}}_{\epsilon_1}|u(s)|_{\cal B}\times \\ &\qquad \times  e^{\iota_t \psi
    (s)-\omega (t-s)}(\hat r_q^*(\omega)^q-|\iota_t \psi
  '(s)|^q)^{\frac{1}{q}}_{\epsilon_2}|u^*(s)|_{\cal
    B^*}ds\\
     & \le
  e^{\omega t}\sup_{[0,t]} \frac{Me^{-\phi-\iota_t \psi }} {(\hat r_p(\omega
    )^p-|\phi '|^p)^{\frac{1}{p}}_{\epsilon_1}(\hat r_q^*(\omega)^q-|\iota_t \psi
    '|^q)^{\frac{1}{q}}_{\epsilon_2}}\times\\ & \qquad \times \|(\hat r_p(\omega
  )^p-|\phi '|^p)^{\frac{1}{p}}_{\epsilon_1}u\|_{p, \omega \cdot -\phi }\|(\hat r_q^*(\omega)^q- |\iota_t \psi
  '|^q)_{\epsilon_2}^{\frac{1}{q}}u^*\|_{\omega (q, t-\cdot )-\iota_t \psi }.
  \end{split}
  \end{equation*}

Using (\ref{L2.6p}), (\ref{L2.7p}) when $\epsilon_j=+$ or   (\ref{L2.6bisp}), (\ref{L2.7bisp}) when $\epsilon_j=-$,   we get
\begin{multline}  \label{eq:estimp=2p}
| <u(t), u^*(t)>_{\cal B,\cal B^*}|  \le  e^{\omega t}\sup_{[0,t]} \frac{Me^{-\phi- \iota_t \psi }} {(\hat r_p(\omega)^p-|\phi '|^p)^{\frac{1}{p}}_{\epsilon_1}(\hat r_q^*(\omega)^q-|\iota_t \psi
    '|^q)^{\frac{1}{q}}_{\epsilon_2}}\times\\   \times 
  \|(\hat r_p(\omega
  )^p-|\phi '|^p)^{\frac{1}{p}}_-m\|_{p,\omega \cdot -\phi  }\|(\hat r_q^*(\omega)^q-|\psi
  '|^q)_-^{\frac{1}{q}}\, m\|_{q,\omega \cdot
-\psi     }|u(0)|_{\cal B}\, |u^*(t)|_{\cal B^*}.
 \end{multline}
  
This estimate holding for any $u^*(t)$, we get
\begin{equation}\label{LL.3p}
\begin{split}
  |u(t)|_{\cal B} \le  e^{\omega t}\sup_{[0,t]} \frac{Me^{-\phi- \iota_t \psi }} {(\hat r_p(\omega)^p-|\phi '|^p)^{\frac{1}{p}}_{\epsilon_1}(\hat r_q^*(\omega)^q-|\iota_t \psi
    '|^q)^{\frac{1}{q}}_{\epsilon_2}}\times\\   \times 
  \|(\hat r_p(\omega
  )^p-|\phi '|^p)^{\frac{1}{p}}_-m\|_{p,\omega \cdot -\phi  }\|(\hat r_q^*(\omega)^q-|\psi
  '|^q)_-^{\frac{1}{q}}\, m\|_{q,\omega \cdot
-\psi     }|u(0)|_{\cal B}\,.
\end{split}
\end{equation} 

 In order to optimize the choice of $M$, we let $0\not\equiv F\in
 C([0,t];[0,+\infty [)$ and study
\begin{equation}\label{LL.4p}
\inf_{0\le M\in C([0,t]),\atop \, \int Mds=1}\sup_s \frac{M(s)}{F(s)}.
\end{equation}
We first notice that
$$
1=\int Mds=\int \frac{M}{F}Fds\le \left(\sup_s \frac{M}{F} \right)\int Fds\,,
$$
and hence
the quantity (\ref{LL.4p}) is $\ge 1/\int Fds$.\\ Choosing $M=\theta F$ 
with $\theta =1/\int F(s)\,ds$, we get equality.

\begin{lemma}\label{LL1p} For any continuous function $F\geq 0$, non identically $0$,
$$
\inf_{0\le M\in C([0,t]),\atop \, \int M(s)\,ds=1}\  \left(\sup_s \frac{M}{F} \right)  = 1/\int Fds\,.
$$
\end{lemma}
Applying the lemma to the
supremum in (\ref{LL.3p})  with  
$$F= e^{\phi+\iota_t \psi}\, (\hat r_p(\omega
    )^p-|\phi '|^p)^{\frac{1}{p}}_{\epsilon_1}(\hat r_q(\omega)^q-|\iota_t \psi
    '|^q)^{\frac{1}{q}}_{\epsilon_2},$$ 
we get 
\begin{equation}\label{LL.5ap}
 |u(t)|_{\cal B} \le  e^{\omega t}\frac{\|(\hat r_p(\omega
  )^p-|\phi '|^p)^{\frac{1}{p}}_-\, m\|_{p,\omega \cdot -\phi  } \|(\hat r_q^*(\omega
  )^q- |\psi '|^q)^{\frac{1}{q}}_-\, m\|_{ q,\omega \cdot -\psi } }
{\int_0^t e^{\phi + \iota_t \psi }(\hat r_p(\omega
    )^p-|\phi '|^p)^{\frac{1}{p}}_{\epsilon_1}(\hat r_q^*(\omega)^q-|\iota_t\psi
    '|^q)^{\frac{1}{q}}_{\epsilon_2} ds}
  |u(0)|_{\cal B}.
\end{equation}
Since $u(0)$ is arbitrary, and noting \eqref{eq:3.7}
we get Theorem \ref{Th1.2p}.\\~\\

\section{Consequences of Theorem \ref{Th1.2p}}
\subsection{Main proposition}
An important step is to prove (we assume $\omega=0$, $\hat r_p(0)=1$) as a consequence of Theorem \ref{Th1.2p} with $\epsilon_1 =+$ and $\epsilon_2=-$, the following key proposition: 
 \begin{proposition}\label{propminmaxp} Assume that $\omega =0$,
   $\hat r_p(\omega )=1$.
  Let $a, b$ positive. Then for  $t\geq a+b$,  
 \begin{equation}\label{eq:optprob}
 || S(t)|| \leq \exp -( t -a - b ) \,\frac{ \left(\inf_u \int_0^a
     m(s)^p (|u'(s)|^p-u^p(s))_+ ds\right)^{1/p}  }{  \left( \sup_\theta \int_0^{b } \frac{1}{m^p}  (\theta(s)^p-|\theta'(s)|^p) \,ds \right)^{1/p} } \,,
 \end{equation}
 where
 \begin{itemize}
 \item $u\in W^{1,p}(]0,a[)$ satisfies $u(0)=0$, $u(a) =1$\,;
 \item $\theta \in W^{1,p}((]0,b[)$ satisfies $\theta(b) = 1$ and $ |\theta '|\le \theta$\,.
 \end{itemize}
 \end{proposition}
 The analysis of the minimizers of 
 \begin{equation}\label{eq:defIp}
 I_{inf,p} := \inf_u \int_0^a
     m(s)^p (|u'(s)|^p-u^p(s))_+ ds
     \end{equation}
      and the maximizers of  
      \begin{equation}
      \label{eq:defJp}
      J_{sup,p}:=\sup_\theta \int_0^{b } \frac{1}{m^p}  (\theta(s)^p-|\theta'(s)|^p) \,ds  
      \end{equation}
       is close to what was done in Section 3 of \cite{HelSj21} and we will sketch what has been modified in Subsection \ref{ss4.3}.
 \subsection{From Proposition \ref{propminmaxp} to Theorem \ref{thBanach}}
 This proposition implies rather directly Theorem \ref{thBanach} in the
 following way. We first observe the trivial lower bound (take
 $\theta(s)=1$)
  \begin{equation}\label{eq:gp2} \sup_\theta \int_0^{b } \frac{1}{m^p}  (\theta(s)^p-|\theta'(s)|^p) \,ds \geq  \int_0^{b } \frac{1}{m^p} ds\,.
  \end{equation}
 A more tricky argument based on the equality case in H\"older's
  inequality\footnote{In \cite{HelSj21} we were using instead Cauchy-Schwarz} gives
\begin{equation}\label{eq:gp3}
 \inf_u \int_0^a m(s)^p (u'(s)^p-u(s)^p)_+ ds  \leq  \inf_u \int_0^a m(s)^p u'(s)^p ds  \leq \Big(1/  \int_0^{a} \frac{1}{m^q} ds\Big) ^{p/q}\,,
 \end{equation}
 More precisely we start from the upper bound 
   $$\int_0^a  (|u'|^p -u^p)_+ m(s)^p ds \leq \int_0^a | u'(s)|^p  m(s)^p \, ds   $$ 
   and minimize the right hand side. \\
   Observing that
   \[
   \begin{split}
   1&=u(a)  = \int_0^a u' (s) \, ds 
   = \int_0^a u' (s)m(s)\, m(s)^{-1}  \, ds\\ &  \leq   \left(\int_0^a (u' (s) m(s)^p ds\right)^{\frac 1p} \,\, \left(\int_0^a \frac{1}{m(s)^q} ds\right)^{\frac 1q}\,,
   \end{split}
   \]
   we look for a $u$ for which we have equality.\\
   By a standard  criterion for the optimality of the H\"older inequality, this is the case if,  for some constant $C>0$,   $$(|u' (s)| m(s))^p =\frac{C}{m(s)^q }\,,$$
   or
  $$
  |u'|=  C^{1/p}   \frac{1}{m^{q} }\,.
  $$
   Hence, we choose 
   $$
   u (s) =\hat C   \int_0^s \frac{1}{m(\tau)^{q} }\, d\tau\,,
   $$
   where the choice of $\hat C$ is determined by imposing $u(a)=1$.
  We obtain
  \begin{lemma}\label{prop3.6}
  For any $a >0$, 
   \begin{equation}
   \inf_{\{u\in W^{1,p}(]0,a[), u(0)=0, u(a)=1\}} \int
_0^a (|u'|^p-u^p)_+m^p\, ds \leq   \left(\int_0^a \frac{1}{m(s)^q}
  ds\right)^{-p/q}\,.
   \end{equation}
   \end{lemma}
Note here that we have no condition on $a>0\,$.
\begin{remark}
Nevertheless, one can observe that this implication holds only under the additional condition that $\mathcal B$ is reflexive.  This was not needed in the direct proof of Theorem \ref{thBanach} given in Section \ref{s2}.
\end{remark}
\subsection{Proof of Proposition \ref{propminmaxp}}\label{ss4.3}
We now  assume $\omega=0$ and $ \hat r_p(0)=1$.
In this case, \eqref{LL.5bis} takes the form
\begin{equation}\label{LL.5ter}
 || S (t)|| _{\mathcal L(\cal B)}\le  e^{\omega t}\frac{\|(\Phi^p -|\Phi '|^p)^{\frac{1}{p}}_-\,m\|_{L^p([0,t[) } \|( \Psi
  ^q- |\Psi '|^q)^{\frac{1}{q}}_-\, m\|_{ L^q([0,t[) }}
{\int_0^t (\Phi^p -|\Phi '|^p)^{\frac{1}{p}}_{+}(\iota_t \Psi^q-|\iota_t \Psi
    '|^q)^{\frac{1}{q}}_{-} ds}\,.
\end{equation}
Replacing $(\Phi, \Psi)$ by $(\lambda \Phi,\mu \Psi)$ give for any $(\lambda,\mu)\in( \mathbb R\setminus \{0\})^2$ does not change the right hand side. Hence we may  choose a suitable normalization without loss of generality. We also choose $\Phi$ and $\Psi$ to be piecewise $C^1([0,t])$.\\
If $0\le \sigma <\tau <+ \infty $, $p>1$,  and $S,T\in \mathbb R$ we put
\begin{equation}\label{red.0.5}W_{S,T}^{1,p}(]\sigma ,\tau [)=\{ u\in W^{1,p}(]\sigma ,\tau [);\, u(\sigma
)=S,\ u(\tau )=T\}\,,\end{equation}
where $W^{1,p}$ denotes the classical Sobolev space associated with $L^p$.\\
Here and in the following all functions are
assumed to be real-valued unless stated otherwise.
As in \cite{HelSj21}  we can here replace $W_{0,1}^{1,p}$ by a subspace that
allows to avoid the use of positive parts. Put
\begin{equation}\label{red.2}
{\cal H}^{a,p}={\cal W}_{0,a}^{0,1,p}\,,
\end{equation}
where for $\sigma ,\tau ,S,T$ as above,
\begin{equation}\label{red.3}
{\cal W}_{\sigma ,\tau }^{S,T,p}=\{u\in W_{S,T}^{1,p}(]\sigma ,\tau [); 0\le
u\le u'\}.
\end{equation}
\begin{equation}\label{red:3ap}
{\cal G}^{b,p}=\{\theta \in W^{1,p}(]0,b[);\, |\theta '|\le \theta,\   \theta
(b)=1 \}\,. 
\end{equation}
Given some $t>a + b $, we now give the conditions satisfied by $\Phi$:
\begin{property}[$P_{a,b}^p$]~
  \begin{enumerate}
  \item $\Phi = e^a u $ on $]0,a]$ and  $u\in \mathcal H^{a,p}$.
  \item On $[a,t - b  ]$, we take  $\Phi(s) =e^{s}$, so
  $|\Phi'|^p(s) -\Phi(s)^p=0\,.$
  \item  On $[t-  b  , t]$ we take $\Phi(s)=e^{t-b }\theta (t-s)$ with $ \theta \in \mathcal G^{b,p}$ \,.
  \end{enumerate}
  \end{property}
   Hence, we have 
  $$
  {\rm supp\,} ( \Phi^p -|\Phi'|^p)_+\subset [ t-b , t]\,.
  $$
 
 Similarly we assume that $\Psi$ satisfies property $(P_{b,a}^q)$ but
 with $\theta=1$,  hence
 \begin{enumerate}
 \item 
 $
 \Psi(s) = e^{b} v(s) \mbox{ on } ]0,b[  \mbox{ with }  v  \in \mathcal H^{b,q}\,,
 $ 
\item    On $[b,t-a  ]$, we take  $\Psi(s) =e^{s}$.
\item On $[t-a ,t ]$,  $\Psi (s) = e^{t-a}\,.$
\end{enumerate}

 Recalling the definition of $\iota_t$, we get for $\iota_t\Psi$:
  \begin{enumerate}
  \item On $[0,a]$, $\iota_{t} \Psi= e^{(t- a )}$, satisfying 
  $$(\iota_t \Psi) '^q -(\iota_t \Psi)^q =  - e^{q(t-a)}\,.$$
  \item On $[a,t -  b  ]$, we have  $\iota_t \Psi (s) =e^{ t-s}$, hence 
  $$|(\iota_t \Psi)'(s)|^q -\iota_t \Psi(s)^q=0\,.$$
  \item  On $]t-   b  , t[$, we have 
  $$|(\iota_t \Psi) '(s)|^q -(\iota_t \Psi)^q(s) \geq  0 \,.
  $$
  \end{enumerate}
 Assuming that $t>a + b $, we have  under these assumptions on $\Phi$ and $\Psi$
$$  
  \{ s;\Phi(s)^p-|\Phi '(s)|^p >0,\iota_t \Psi(s)^q -|\iota_t \Psi '(s)|^q < 0  \} \subset [ t-b,b]\,.
  $$
 We now compute or estimate the various quantities appearing in \eqref{LL.5bis}.\\

We have
\begin{equation}\label{eq:Iphi}
  \| (\Phi^p-|\Phi'|^p)^{\frac{1}{p}}_-\,m\| = e^{a}\left(  \int_0^a
    (|u'(s)|^p - u^p(s)) m(s)^p ds
  \right)^{1/p}\,,
\end{equation}
\begin{equation}\label{eq:Ipsi}
\|(\Psi^q-|\Psi '|^q)^{\frac{1}{q}}_-m\| =  e^{b}\left( \int_0^{b }
  (|v'(s)|^q - v(s)^q) m(s)^q ds\right)^{1/q}
  \,,
\end{equation} 
and  
\begin{equation}\label{eq:Ichi}
\begin{split}
&\hskip -1cm\int_0^t(\Phi^p-\Phi'^p)^{\frac{1}{p}}_+ ((\iota_t \Psi)^q- (\iota_t \Psi')^q )^{\frac{1}{q}}_-ds \\ 
&=  \int_{t-b} ^t(\Phi^p-|\Phi'|^p)^{\frac{1}{p}}_+ ((\iota_t \Psi)^q-
  |\iota_t \Psi'|^q)^{\frac{1}{q}}_-ds \\ 
 &= e^{t-b } \int_{t-b }^t   (\theta(t-s) ^p- |\theta'(t-s)|^p)_+^{\frac 12}  \,   ((\iota_t \Psi)^q- |\iota_t \Psi'|^q)^{\frac{1}{q}}_- ds
\\ 
& = 
  e^{t} \int_{0}^{b }   (\theta (s) ^p- |\theta'(s)|^p)^\frac 1p (|v'(s)|^q - v (s)^q)^{\frac{1}{q}} ds \,.
  \end{split}
\end{equation}
So we get from \eqref{LL.5bis}
\begin{equation}\label{LL.5z}
 || e^t  S (t)|| _{\mathcal L(\cal B)}\le  e^{a+ b}   \left (\int_0^a (|u'(s)|^p - u^p(s)) m(s)^p ds\right)^\frac 1p \, K^p (b,\theta,v) \,,
\end{equation}
where 
\begin{equation}\label{Kbthetav}
 K^p (b,\theta,v) := \frac{ \left(\int_0^{b } (|v'(s)|^q - v(s)^q) m(s)^q
     ds\right)^{1/q} }{ \int_{0}^{b}   (\theta (s)^p -| \theta'(s)|^p)^\frac 1p  (|v'(s)|^q - v (s)^q)^{\frac{1}{q}} ds }\,.
 \end{equation}
 We start by considering for a given $\theta \in \mathcal G^{b,p}$
 $$
K_{\mathrm{inf}}^p (b,\theta) := \inf_{ v \in \mathcal H^{b,q}} K^p(b,\theta ,v)\,,
 $$
and  get  the following:
  \begin{lemma}\label{prop3.12} If $\theta\in \mathcal G^{b,p}$ and $\theta
    -\theta '$ is not identically $0$ on $]0,b$[, we have
\begin{equation}\label{idkinfbtheta}
K^p_{\mathrm{inf}} (b,\theta)  = \frac{1} {\big( \int_0^{b }   (\theta (s) ^p- |\theta'(s)|^p) \frac {1}{ m^p} ds\big)^{1/p}  } \,.
\end{equation}
\end{lemma}
\begin{proof}
We consider
with $$h(s)=  \big(\theta (s) ^p- |\theta'(s)|^p\big) ^\frac 1p  \geq 0 $$ the
denominator in (\ref{Kbthetav}), 
$$
\int_0^b h(s)  (|v'(s)| ^q - v(s)^q)^{\frac{1}{q}} ds =\int_0^b \big( h(s)/m(s) \big) \,\big(   m^q(|v'(s)| ^q- v(s)^q)^{\frac{1}{q}}\big) \,ds  \,.
$$
By the H\"older inequality, we have
\begin{multline*}
\int_0^b h(s) (|v'(s)|^q - v(s)^q)^{\frac{1}{q}} ds \leq \\ \left(
  \int_0^b ( h(s)/m(s) )^pds\right)^\frac{1}{p} \, \left(\int_0^b m(s)^q(|v'(s)|^q -
  v(s)^q) ds \right) ^\frac 1q \,,
\end{multline*}
which implies that $K^p_{\mathrm{inf}}(b,\theta )$ is
  bounded from below by the right hand side of (\ref{idkinfbtheta}).

 We have equality for some $ v$ in $\mathcal H^{b,q}$ if and only if,  for some constant $c>0$, 
$$
 m(s)^q (|v'(s)|^q - v(s)^q) = c \,  h(s)^p/m(s)^p  \,.
 $$
 In order to get such a $v$,
we first consider $w \in W^{1,p}$ defined by
$$
w' = ( w^q  + h^pm^{-pq})^{1/q}\,,\, w(0)=0\,,
$$
noticing that the right hand side of the differential
  equation is Lipschitz continuous in $w$, so that the
  Cauchy-Lipschitz theorem applies.
According to our assumption on $\theta$, we verify that $w(b) >0$ 
 and we choose
 $$
v = \frac{1}{w(b)} w\,,\,c = \frac{1}{w(b) }\,.
$$
For this pair $(c,v)$ we get
\begin{equation}\begin{split}
&\left(\int_0^{b } (|v'(s)|^q - v(s)^q) m(s)^q ds\right)^\frac 12 \Big/
  \int_{0}^{b }  \left(\theta (s) ^p- |\theta '(s)|^p \right) ^\frac 12
  \left(|v'(s)|^q - v(s)^q\right)^{\frac{1}{q}} ds\\
  &=1 \big/\left( \int_0^b h^p(s)  m^{-p}(s)  \, ds \right)^\frac 1p   \,.
\end{split}
\end{equation}
Returning to the definition of $h$ shows that
  $K^p_\mathrm{inf}(b,\theta )$ is bounded from above by the right hand
  side of (\ref{idkinfbtheta}) and we get the announced result.
\end{proof}
 
 To conclude the proof of  Proposition \ref{propminmaxp}, we just
   combine Lemma \ref{prop3.12}  and   \eqref{LL.5z}.

\subsection{Application of Proposition \ref{propminmaxp}: Wei's theorem}
When $p=2$, we can directly apply \cite{HelSj21} where $r(\omega)$ is replaced by $\hat r_2(\omega)$. The case $p\neq 2$ is less clear. Already, in the case $m=1$ this involves new questions related to the $p$-Laplacian instead of the Laplacian. Nevertheless, one can get (probably non optimal)  upper bounds by using the optimizers obtained for $p=2$. If we consider $m=1$,  \eqref{eq:optprob}, $a \leq \frac \pi 4$, $b \leq \frac \pi 4$ and
take therein $ u(s) = \sin s/\sin a$ and $\theta (s) =  \cos s/\cos b$, we obtain 
 \begin{equation}\label{eq:optprobapp}
 || S(t)|| \leq \cos b \sin a^{-1} \exp -( t -a - b ) \,\frac{ \left(\int_0^a
      ((\cos s)^p-(\sin s)^p) ds\right)^{1/p}  }{  \left(\int_0^{b }( (\cos s)^p-(\sin s)^p) \,ds \right)^{1/p} } \,,
 \end{equation}
  When $a=b$, we obtain
   \begin{equation}\label{eq:optprobapp}
 || S(t)|| \leq \cot a  \exp -( t -2a) \,.
 \end{equation}
 For $a=b=\frac \pi 4$ we get an extension of Wei's theorem to the reflexive Banach case.
 \begin{theorem}\label{th3.2p}
 Let $p>1$ and $S(t)$ a $C_0$-semigroup of generator $A$ in a reflexive Banach space $\mathcal B$
  such that 
   \begin{equation}\label{eq:h1w}
\|S(t)\| \leq 1  \mbox{ for all } t \geq 0\,.
\end{equation}
 holds and $\hat r_p(0) >0$.  Then we have, 
\begin{equation}\label{eq:w}
||S(t) || \leq e^{- \hat r_p(0)t +\frac \pi 2}\,,\, \forall t\geq 0\,.
\end{equation}
\end{theorem}

  \section{Modified Riccati equation and application to the optimization problems.}
When analyzing the optimality of the statements in \cite{HelSj21}, an important tool
 was a fine analysis of a natural Riccati equation (see more precisely Subsection~3.4.2 in \cite{HelSj21}). Let us see what is going on for $p\in (1,+\infty)$ and we start with the assumption that
 $$
 \hat r_p(0)=1\,.
 $$\\
 Let $f$ be defined on $]\sigma ,\tau [$  such
that
\begin{equation}\label{mh.8}0<f\le f'.\end{equation}
Put $$\mu =m'/m\,.$$
We now assume that $f$ satisfies  $$(\partial _s\circ m^p\circ \partial
_s+m^p)f^{p-1}=0\,,$$
we get
\begin{equation}\label{mh.9}
(\partial _s^2+p \mu \partial _s+1)f^{p-1}=0\,.
\end{equation}
In this case, we say that $f$ is $(m,p)$-harmonic.\\
Writing $$\phi =\log  f \mbox{ and }  \psi =\phi '=f'/f\,,$$
we get, noting that
$$
f''/f= \psi' +\psi^2\,,\,  \psi \ge 1 
$$
and 
\begin{equation}\label{mh.10}
    \psi '=-((p-1)\psi ^2+p\mu \psi +\frac{1}{p-1})\,,
\end{equation}
or equivalently
\begin{equation}\label{mh.10bis}
   \frac{ \psi '}{\psi} =-((p-1)\psi +p\mu  +\frac{1}{(p-1)\psi })\,.
\end{equation}
As in \cite{HelSj21}, we note that $\tilde \psi := - 1/\psi$
satisfies
\begin{equation}\label{mh.10ter}
   \frac{ \tilde \psi '}{\tilde \psi} =-((p-1)\tilde \psi - p\mu  +\frac{1}{(p-1)\tilde \psi })\,.
\end{equation}

We consider the condition 
\begin{equation}\label{mh.10p}
\lim_{x \mapsto 0^+}\psi_p(x) = +\infty\,.
\end{equation}
With the conditions \eqref{mh.10} and \eqref{mh.10p}, $\psi=\psi_{m,p}$ is uniquely defined and 
and we introduce 
 \begin{equation}\label{defa*}
a^*=a^*(m,p)= \sup \{ a>0 \mbox{ and } \psi_{m,p} >1 \mbox{ on }(0,a)\}\,,
\end{equation}
so that $a ^*(m,p)\in ]0,+\infty ]$.\\
Following \cite{HelSj21}, one can prove that in \eqref{eq:defIp}  the infimum is realized by  the $(m,p)$-harmonic function  $u_p$ such that $u'_p/u_p=\psi_p$.\\
Hence we get with in mind that in \eqref{eq:defIp} the infimum is for $u\in \mathcal H^{a,p}$
 \begin{equation*}
 I_{inf,p} = \int_0^a
    ( m(s)^p u_p'(s)^p-u_p^p(s)) ds
     \end{equation*}
After an integration by parts, we get (since $u_p(0)=0$ and $u_p(a)=1$)
 \begin{equation*}
 I_{inf,p} := m(a)^p (u_p'(a))^{p-1} (a) u_p^{p}(a) = m(a)^p \psi_p^{p-1} (a)\,.
     \end{equation*}
Similarly  one can prove that in \eqref{eq:defJp}  the supremum is realized by  the $(1/m,p)$-harmonic function  $\theta_p\in \mathcal G^{b,p}$ such that $\theta'_p/\theta_p=- 1/\psi_p$.\\
Hence we get, with in mind that in \eqref{eq:defJp} the infimum is for $\theta \in \mathcal G^{b,p}$ satisfying $\theta'(0)=0$, 
  \begin{equation}
      \label{eq:calJp}
      J_{sup,p}= \int_0^{b } \frac{1}{m(s)^p}  (\theta_p(s)^p-|\theta_p'(s)|^p) \,ds  
      \end{equation}
      After an integration by parts, we get (since $\theta'_p(0)=0$ and $\theta_p(b)=1$)
 \begin{equation*}
 J_{sup,p} =- m(b)^{-p} |\theta_p'(b)|^{p-1} (a) \theta_p^{p}(b) = m(b)^{-p} \psi_p^{-(p-1)} (b)\,.
     \end{equation*}
   
\section{Final theorem} 
Like in \cite{HelSj21} and coming back to Proposition \ref{propminmaxp}, we immediately get from the previous section:
    \begin{proposition}\label{prop4.9} Let $p>1$, $\omega =0$, $\hat r_p(0)=1$ and  $a^*:=a^*(m,p)\in ]0,+\infty]$. When
    $a,b\in ]0,+\infty [\cap ]0,a^*]$
  and $t> a+b$, we have 
 \begin{equation}
 || e^t S(t)|| \leq \exp (a +b ) m(a) m(b) \psi_p(a)^\frac{p-1}{p} \psi_p (b)^\frac{p-1}{p}  \,.
\end{equation}
In particular, when $a^*<+\infty $, we have 
 \begin{equation}\label{eq:thintro}
 || e^t S(t)|| \leq \exp (2 a^{*})   \, m(a^{*} )^2 \,,\ \ t>2 a^*\,.
 \end{equation}
\end{proposition}
 This proposition is the analog of Wei's theorem for general  weights $m$ and the $L^p$-Banach version of Theorem 1.9  in \cite{HelSj21}.\\
 By the same rescaling  procedure, we
 have actually a more general statement.  We consider $\hat A$ with the same properties as $A$ where the hat's are introduced to make easier the transition between the particular case above to the general case below. As before, we introduce
   $\hat \omega $ and $\hat r =\hat r_p(\hat \omega)$. 
    \begin{theorem}\label{prop4.9bis} 
    Let $p>1$,  $\hat r_p(\hat \omega)< +\infty$. 
    Let $\hat S(\hat t)= e^{\hat t \, \hat A} $ satisfying $$
    || \hat S (\hat t) || \leq \hat m(\hat t)\,,\, \forall \hat t >0\,.
    $$
    Then there exist uniquely defined  $\hat a^*:=\hat a^*(\hat m, \hat \omega,\hat r,p )>0 $ and $\hat \psi_p:=\hat \psi_p(\cdot ; \hat m, \hat \omega,\hat r )$ on $]0,\hat a^*[$ with the same general properties as above  such that,
 if  $\hat a, \hat b\in ]0,+\infty [\cap ]0,\hat a^*]$
  and $\hat t> \hat a+\hat b$, we have 
 \begin{equation} 
 || S(\hat t)|| \leq \exp \left( ( \hat \omega  -  \hat r_p(\hat \omega)) (\hat t- (\hat a + \hat b )) \right)  \hat m(\hat a) \hat m(\hat b)   \hat \psi_p(\hat a)^\frac{p-1}{p} \, \hat \psi_p(\hat b )^\frac{p-1}{p}  \,.
\end{equation}
Moreover, when $ \hat a^*<+\infty $, the estimate is optimal for $\hat a=\hat b=\hat a^{*} $ and reads 
 \begin{equation}\label{eq:thintrobis}
 ||  \hat S(\hat t)|| \leq \exp ((\hat \omega   -  \hat r_p(\hat \omega)  (\hat t - 2\hat a^{*}))   \, \hat m(\hat a^*)^2 \,,\ \ t>2 \hat a^*\,.
 \end{equation}
 \end{theorem}
 Note that in the statement
 $$
 \hat a^* (\hat m,\hat \omega)= \hat r \, a^* (e^{-\hat \omega \cdot}\hat m) \,,\, \hat \psi_p(\hat s;  \hat m, \hat \omega,\hat r ) = \psi_p (\hat r \hat s ;  e^{-\hat \omega \cdot}\hat m)   \,.
$$
This theorem is the $L^p$-Banach version of Theorem 1.10  in \cite{HelSj21}.\\

\end{document}